\documentclass[a4paper,11pt]{article}
\usepackage[cp1251]{inputenc}
\usepackage[french, english]{babel}
\usepackage{amsfonts}
\usepackage{amsthm}
\usepackage{amsbsy}
\usepackage{amssymb}
\usepackage{graphicx}
\usepackage{eso-pic}
\usepackage{dsfont}
\usepackage{tikz}
\usepackage{xfrac}
\usepackage{float}
\usepackage[bookmarks=true]{hyperref}
\usepackage{bookmark}
\usepackage{amsmath}
\usepackage{subfigure}
\usepackage{hyperref}
\usetikzlibrary{matrix,arrows}

\usepackage[letterpaper,left=0.8in,right=0.8in,top=1in,bottom=1in]{geometry}
\newtheorem{lemma}{Lemma}[section]
\newtheorem{teo}[lemma]{Theorem}
\newtheorem{prop}[lemma]{Proposition}

\theoremstyle{definition}
\newtheorem{defn}[lemma]{Definition}

\newtheorem{quest}[lemma]{Question}

\newtheorem{example}[lemma]{Example}

\theoremstyle{remark}
\newtheorem{rem}[lemma]{Remark}

\newcommand{\calX} {\ensuremath {\mathcal{X}}}
\newcommand{\calR} {\ensuremath {\mathcal{R}}}

\newcommand{\calS} {\ensuremath {\mathcal{S}}}
\newcommand{\calM} {\ensuremath {\mathcal{M}}}
\newcommand{\calC} {\ensuremath {\mathcal{C}}}

\newcommand{\calB} {\ensuremath {\mathcal{B}}}

\newcommand{\calT} {\ensuremath {\mathcal{T}}}
\newcommand{\calY}{\ensuremath {\mathcal{Y}}}

\newcommand{\calO}{\ensuremath {\mathcal{O}}}

\newcommand{\calZ}{\ensuremath {\mathcal{Z}}}

\begin{document}
\title{The complement of the figure-eight knot geometrically bounds}
\author{Leone Slavich}
\date{}
\maketitle

\begin{abstract}
\noindent We show that some hyperbolic $3$-manifolds which are tessellated by copies of the regular ideal hyperbolic tetrahedron are \emph{geodesically embedded} in a complete, finite volume, hyperbolic $4$-manifold. This allows us to prove that the complement of the figure-eight knot \emph{geometrically bounds} a complete, finite volume hyperbolic $4$-manifold. This the first example of geometrically bounding hyperbolic knot complement and, amongst known examples of geometrically bounding manifolds, the one with the smallest volume.
\end{abstract}

\section{Introduction}

An orientable, complete, finite volume hyperbolic $n$-manifold \emph{geometrically bounds} if it is realised as the totally geodesic boundary of an orientable, complete, finite volume, hyperbolic $(n+1)$-manifold. 
The problem of understanding which hyperbolic $3$-manifolds bound geometrically hyperbolic $4$-manifolds dates back to work of Long and Reid \cite{longreid}, \cite{longreid1}. This problem is related to physics, in particular, to the theory of hyperbolic gravitational instantons, as shown in \cite{gibbons}, \cite{ratcliffetschanz} and \cite{ratcliffetschanz3}.

Geometrically bounding manifolds exist in all dimensions \cite{longreid1}, however it is shown in  \cite{longreid} that many closed hyperbolic $3$-manifolds do not bound geometrically compact hyperbolic $4$-manifolds: the property of being a geometric boundary is, at least in the compact case, non-trivial.

The first examples of closed geometrically bounding hyperbolic $3$-manifolds were constructed by Ratcliffe and Tschantz in \cite{ratcliffetschanz}. An infinite family of closed, geometrically bounding manifolds was constructed by Kolpakov, Martelli and Tschantz in \cite{martelli4}. The smallest known example of closed geometrically bounding manifold has volume $68.8992\dots$ (see \cite{martelli4}).

The first examples of non-compact, geometrically bounding manifolds were constructed in \cite{GPS}. More examples are described in \cite{ratcliffetschanz}. The first non-compact example given by the complement of a link in the $3$-sphere was built in \cite{Slavich}. Other examples of geometrically bounding link complements were constructed in \cite{KS2014} and \cite{martelli}. In particular, in \cite{martelli} it is shown that the complement of the Borromean rings geometrically bounds. With volume $7.32772\dots$, this was the smallest known geometrically bounding non-compact manifold, as well as the one with the smallest number of cusps.  It is also worth mentioning that all known examples of geodesically bounding hyperbolic link complements are arithmetic and have each of their respective components unknotted.

A very closely related notion is the one of \emph{geodesically embedded} manifold: a hyperbolic $n$-manifold $M$ is geodesically embedded if it is realised as a totally geodesic submanifold of a complete, finite volume hyperbolic $(n+1)$-manifold $\calX$. It is clear that if a hyperbolic manifold $M$ bounds geometrically an $(n+1)$-manifold $\calY$, then it is also geodesically embedded (simply mirror the manifold $\calY$ in its boundary). 
On the other hand, if the manifold $\calM$ it admits a fixed-point-free orientation reversing involution $i$ and is geodesically embedded in $\calX$, then it bounds geometrically, as proven in Lemma \ref{lemma:bound}.

In \cite{martelli} it is shown that all hyperbolic $3$-manifolds which are tessellated by the regular ideal hyperbolic octahedron or by the right-angled hyperbolic dodecahedron are geodesically embedded. The puropose of this paper is to build examples of non-compact, geodesically embedded hyperbolic $3$-manifolds which are tessellated by copies of the regular ideal hyperbolic tetrahedron. We call manifolds which admit such a tessellation \emph{tetrahedral}. Notice that in an ideal triangulation of a tetrahedral manifold, the valence of an ideal edge is necessarily $6$. 
It is worth mentioning that a tetrahedral $3$-manifold $M$ can have many non-equivalent triangulations by regular ideal hyperbolic tetrahedra. A complete classification up to homeomorphism of orientable tetrahedral manifolds which possess a triangulation with at most $25$ tetrahedra ($21$ tetrahedra in the non-orientable case) is given in \cite{FGGTV}, where it is also shown that all tetrahedral manifolds are arithmetic.

One of the most famous examples of tetrahedral manifold is the complement of the figure-eight knot, represented in Figure \ref{fig:fig8}.

\begin{figure}[htbp]
\centering
\includegraphics[width=0.25\textwidth]{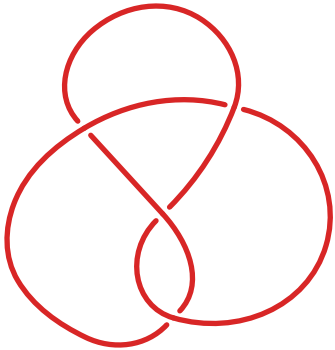}
\caption{The figure-eight knot}\label{fig:fig8}
\end{figure}

The complement of the figure-eight knot is hyperbolic and it admits an ideal triangulation consisting of two regular ideal hyperbolic tetrahedra \cite[Chapter 1]{notes}. Its hyperbolic volume is equal to $2.029883\dots$, which is the minimal volume for a hyperbolic knot complement. Moreover, it fibers over the circle with fiber given by a once-punctured torus, and it is the only arithmetic knot complement \cite{reid}.

It is reasonable to ask whether the complement of the figure-eight knot geometrically bounds. The main result of this paper is a positive answer to such question:

\begin{teo}\label{teo:fig8bounds}
The complement of the figure-eight knot geometrically bounds an orientable, complete, finite volume, hyperbolic $4$-manifold.
\end{teo}

The paper is organized as follows: in Section \ref{sec:5-cell} we introduce the main ingredient of our construction, a $4$-dimensional hyperbolic Coxeter polytope called the \emph{rectified $5$-cell} $\calR$. This polytope is obtained by truncating the $4$-dimensional simplex $S^4$. As a consequence of this fact it is possible to encode gluings between copies of $\calR$ along certain facets using $4$-dimensional triangulations. 

In Section \ref{sec:triangulations} we introduce the notion of a \emph{$6$-valent $4$-dimensional triangulation with trivial face cycles}, and we prove that such a triangulation defines a hyperbolic $4$-manifold with totally geodesic boundary. 
We show that the boundary components of the manifold are in a one-to-one correspondence with the vertices of the triangulation, and that the geometry of each boundary component is encoded by the link of the associated vertex. 
A straightforward consequence is that all tetrahedral manifolds which can be realised as vertex links of $6$-valent $4$-dimensional triangulations with trivial face cycles are geodesically embedded. 

In Section \ref{sec:fig8bounds} we describe a $6$-valent triangulation which realises the figure-eight knot complement as a vertex link, proving that this knot complement geometrically bounds.

\medskip

\textbf{Acknowledgements.} The author is grateful to the department of Mathematics of the University of Bologna for the hospitality when this work was conceived and written. Also, the author is grateful to Matthias Goerner and Stavros Garoufalidis for several helpful suggestions and remarks, in particular for highlighting the fact that Proposition \ref{prop:manifold} is necessary, and for making the author aware of Remark \ref{rem:minvolume}.

\section{The rectified $5$-cell}\label{sec:5-cell}

Below, we describe the main building ingredient of our construction, \textit{the rectified $5$-cell}, which can be realised as a non-compact finite-volume hyperbolic $4$-polytope. First, we start from its Euclidean counterpart, which shares the same combinatorial properties.  

\begin{defn}\label{defn:5cell}
The \emph{Euclidean rectified $5$-cell} $\mathcal{R}$ is the convex hull in $\mathbb{R}^5$ of the set of $10$ points whose coordinates are obtained as all possible permutations of those of the point $(1,1,1,0,0)$.
\end{defn}

The rectified $5$-cell has ten facets ($3$-dimensional faces) in total. Five of these are regular octahedra. They lie in the affine planes defined by the equations 
\begin{equation}
\sum_{i=1}^5 x_i=3,\; x_j=1, \mbox{ for each } j\in \{1,2,3,4,5\},
\end{equation}
and are naturally labelled by the number $j$. 

The other five facets are regular tetrahedra. They lie in the affine hyperplanes given by the equations 
\begin{equation}
\sum_{i=1}^5 x_i=3,\; x_j=0, \mbox{ for each } j\in \{1,2,3,4,5\},
\end{equation}
and are also labelled by the number $j$.

Also, the polytope $\mathcal{R}$ has $30$ two-dimensional triangular faces, $30$ edges and $10$ vertices. The triangular faces are of two types: $10$ of these are adjacent to two octahedral facets, while the other $20$ are adjacent to a tetrahedron on one side and to an octahedron on the other side.
We color the $2$-faces of the first type in red, and those of the second type in blue.

We note the following facts about the combinatorial structure of $\mathcal{R}$:
\begin{enumerate}
\item each octahedral facet $F$ has a red/blue chequerboard colouring, such that $F$ is adjacent to any other octahedral facet along a red face, and to a tetrahedral facet along a blue face;
\item a tetrahedral facet having label $j\in \{1,2,3,4,5\}$ is adjacent along its faces to the four octahedra with labels $k\in \{1,2,3,4,5\}$, with $k$ different from $j$;
\item the tetrahedral facets meet only at vertices and their vertices comprise all those of $\mathcal{R}$.
\end{enumerate}

\begin{defn}\label{defn:hyperbolic-5-cell}
Like any other uniform Euclidean polytope, the rectified $5$-cell has a hyperbolic ideal realisation, which may be obtained in the following way:
\begin{enumerate}
\item normalise the coordinates of the vertices of $\mathcal{R}$ so that they lie on the unit sphere $\mathbb{S}^3\subset \mathbb{R}^4$;
\item interpret $\mathbb{S}^3$ as the boundary at infinity of the hyperbolic $4$-space $\mathbb{H}^4$ in the Klein-Beltrami model.
\end{enumerate}
The convex hull of the vertices of $\mathcal{R}$ now defines an ideal polytope in $\mathbb{H}^4$, that we call the \emph{ideal hyperbolic rectified $5$-cell}.
\end{defn}

With a slight abuse of notation, we continue to denote the ideal hyperbolic rectified $5$-cell by $\mathcal{R}$.

\begin{rem}
In the hyperbolic realization of $\calR$, the octahedral facets become regular ideal right-angled octahedra, while the tetrahedral facets become regular ideal hyperbolic tetrahedra. The traingular faces become ideal triangles, and all the edges are ideal, meaning that they connect two ideal vertices.
The vertex figure $P$ of the ideal hyperbolic rectified $5$-cell is a right Euclidean prism over an equilateral triangle, with all edges of equal length. At each vertex, there are three octahedra meeting side-by-side, corresponding to the square faces, and two tetrahedra, corresponding to the triangular faces.

The dihedral angle between two octahedral facets is therefore equal to $\pi/3$, while the dihedral angle between a tetrahedral and an octahedral facet is equal to $\pi/2$. 
\end{rem}

\begin{rem}\label{rem:volume}
The volume $v_{\mathcal{R}}$ of the rectified $5$-cell equals $2\pi^2/9$, as shown in \cite{KS2014}.
\end{rem}

\begin{rem}\label{rem:correspondence} Another way to construct the rectified $5$-cell is to start with a regular Euclidean $4$-dimensional simplex $S_4$ and take the convex hull of the midpoints of its edges. This is equivalent to truncating the vertices of $S_4$, and enlarging the truncated regions until they become pairwise tangent along the edges of $S_4$.

With this construction, it is easy to see that the symmetry group of $\mathcal{R}$ is isomorphic to the symmetry group of $S_4$, which is known to be $\mathfrak{S}_5$, the group of permutations of a set of five elements. Moreover, we obtain the following one-to-one correpondences:
\begin{enumerate}
\item $\{\mbox{Vertices of }\calR\}\leftrightarrow\{\mbox{Edges of }S_4\}$
\item $\{\mbox{Tetrahedral facets of }\calR\}\leftrightarrow\{\mbox{Vertices of }S_4\}$
\item $\{\mbox{Octahedral facets of }\calR\}\leftrightarrow\{\mbox{Facets of }S_4\}$
\item $\{\mbox{Ideal triangles adjacent to two octahedral facets of }\calR\}\leftrightarrow\{\mbox{Triangular faces of }S_4\}$.
\end{enumerate}

\end{rem}

\section{Triangulations and hyperbolic $4$-manifolds}\label{sec:triangulations}

As a consequence of the correspondences between the strata of $\calR$ and the strata of $S_4$, we can encode glueings between copies of $\calR$ using $4$-dimensional triangulations.

\begin{defn}\label{defn:triangulation}
A \textit{$4$-dimensional triangulation} $\mathcal{T}$ is a pair 
\begin{equation}
(\{\Delta_i\}_{i=1}^{2n}, \{g_j\}_{j=1}^{5n}),
\end{equation} 
where $n$ is a positive natural number, the $\Delta_i$'s are copies of the standard $4$-dimensional simplex $S_4$, and the $g_j$'s are a complete set of simplicial pairings between the $10n$ facets of all $\Delta_i$'s.
\end{defn}

\begin{defn}\label{defn:orientable-triangulation}
A triangulation is \textit{orientable} if it is possible to choose an orientation for each tetrahedron $\Delta_i$, $i=1,\dots,2n$, so that all pairing maps between the facets are orientation-reversing (see also \cite[Definition $4.2$]{KS2014}).
\end{defn}

Given a $4$-dimensional triangulation $\calT=(\{\Delta_i\}_{i=1}^{2n}, \{g_j\}_{j=1}^{5n})$, we can build a $4$-dimensional CW-complex $M_{\calT}$ as follows:
\begin{enumerate}
\item associate to each $\Delta_i$ a copy $\mathcal{\calR}_i$ of the ideal rectified $5$-cell \calR;
\item a face pairing $g_{kl}$ between the facets $F$ and $G$ of the simplices $\Delta_k$ and $\Delta_l$ defines a \textit{unique} isometry between the respective octahedral facets ${\calO}_F$ and ${\calO}_G$ of $\mathcal{R}_k$ and $\mathcal{R}_l$. The isometry is determined by the behaviour of the pairing map $g_{kl}$ on the edges of the tetrahedra $F$ and $G$, which are in a one-to-one correspondence with the ideal vertices of ${\calO}_F$ and $\calO_G$, respectively.   
\item identify all octahedral facets of the polytopes $\mathcal{R}_i$, $i=1\dots,2n$ using the isometries defined by the pairings $g_j$, $j=1\dots 5n$, to produce $M_{\mathcal{T}}$.
\end{enumerate}

With this construction, the tetrahedral facets of the various copies of $\calR$ are glued together along their triangular faces to produce the boundary $\partial M_{\calT}$ of $M_{\calT}$. If the triangulation $\calT$ is orientable, then also $M_{\calT}$ is orientable. 

Given a $4$-dimensional triangulation $\mathcal{T}=(\{\Delta_i\}_{i=1}^{2n}, \{g_j\}_{j=1}^{5n})$, let us consider the abstract graph with vertices given by the $20n$ two-dimensional faces of the simplices $\{\Delta_i\}_{i=1}^{2n}$ and edges connecting two vertices if the corresponding two-faces are identified by a pairing map. This graph is a disjoint union of cycles $\{c_1,\dots,c_k\}$, which we call the \emph{face cycles} corresponding to the triangulation $\mathcal{T}$.

To each face cycle $c$ of $\calT$ there is a naturally associated affine \emph{return map} $r_c$ from the $2$-simplex to itself: simply follow the simplicial pairings from one simplex to the following one, until the cycle closes up.

A simple condition on the return maps ensure that the complex $M_{\calT}$ is a manifold as expressed in the following Proposition:

\begin{prop}\label{prop:manifold}
Let $\calT$ be a $4$-dimensional triangulation. The complex $M_{\calT}$ is a manifold if and only if, for each face cycle $c$ of $\calT$, the return map $r_c$ is the identity.
\end{prop}

\begin{proof}
We need to check that the links of the midpoints of the $2$-faces of $M_{\calT}$ are spheres, and that the links of the midpoints of the edges are disks (notice that all edges in $\calR$ belong to some tetrahedra, and therefore all the edges of $M_{\calT}$ lie in the boundary). The link of the midpoint of a $2$-face $F$ corresponding to a face cycle $c$ has the structure of a (possibly non-orientable) Seifert bundle over $S^2$, which is built as follows:
\begin{enumerate}
\item Consider the product $S_2\times I$, where $S_2$ denotes the $2$-simplex. 
\item Identify $S_2\times\{0\}$ to $S_2\times\{1\}$ by applying the return map $r_c$ associated to the face cycle $c$, to produce a solid torus or Klein bottle with a model Seifert fibration. 
\item Collapse each fiber in the boundary to a single point. 
\end{enumerate}
The resulting space is a sphere if and only if the return map is the identity (if the return map is a cyclic permutation of the vertices, it is homeomorphic to the lens space $L(3,1)$).

Concerning the edges of $M_{\calT}$, it is sufficient to check that the link of their midpoints in the tetrahedral boundary $\partial M_{\calT}$ are all $2$-spheres or, equivalently, that the return maps for the \emph{edge cycles} in $\partial M_{\calT}$ are trivial. Each edge of $R$ corresponds to a pair $(F,v)$, where $F$ is a $2$-face of $S_4$ and $v$ is a vertex of $S_4$ adjacent to $F$. Because of this, the edge cycles of $\partial M_{\calT}$ are simply the restriction of the face cycles of $\calT$ to the edges of $S_2$ and clearly if the return maps on the face cycles are trivial, then the return maps on the edge cycles are trivial.  

\end{proof}

Now, let us suppose that $\calT$ is a triangulation with trivial return maps.
The manifold $M_{\calT}$ will not be, in general, complete hyperbolic. However, a simple condition on the face cycles of $\calT$ esures that the  hyperbolic structure on all copies of $\calR$ match together to give a hyperbolic structure on $M_{\calT}$.

\begin{defn}
A $4$-dimensional triangulation $\calT$ is \emph{$6$-valent} if \emph{all} the face cycles have length $6$.
\end{defn}

\begin{prop}\label{prop:hyperbolicstructure}
Let $\calT=(\{\Delta_i\}_{i=1}^{2n}, \{g_j\}_{j=1}^{5n})$ be a $6$-valent $4$-dimensional triangulation with trivial return maps. The associated manifold $M_{\calT}$ is a complete, non-compact, finite volume hyperbolic $4$-manifold with totally geodesic boundary.
\end{prop}

\begin{proof}
The pairing maps $\{g_j\}_{j=1}^{5n}$ induce identifications between the square faces of the $2n\cdot 10$ copies $\{P_1,\dots,P_{20n}\}$ of the vertex figure $P$ of $\calR$.
Following \cite{notes}, it is sufficient to check that the resulting complex $\calC$ is a (possibly disconnected) compact Euclidean manifold with totally geodesic boundary. 

Let us notice that the pairings maps preserve the red/blue coloring on the $2$-dimensional faces of $\calR$. Beacuse of this fact, edges of the vertex figure $P$ that separate square faces (which correspond to red $2$-faces of $\calR$) are paired together, and the same holds for the edges separating a square face and a triangular face (which correspond to blue $2$-faces of $\calR$). This implies that the resulting complex $\calC$ is an $I$-bundle over a surface $\calS$ which is tessellated by copies of a Euclidean equilateral triangle, with the corresponding $\partial I$ sub-bundle tessellated by the triangular faces of the vertex figure $P$. 

The existence of a Euclidean structure on $\calC$ is then a consequence of the existence of a Euclidean structure on $\calS$, which is guaranteed if all the vertices of the tessellation of $\calS$ are adjacent to $6$ triangles. Each such vertex correponds to an ideal edge of $P$ separating two square faces, and each such edge correponds to a $2$-stratum of $S_4$. Therefore the hyperbolicity of $M_{\calT}$ is guaranteed if the triangulation $\calT$ is $6$-valent. The totally geodesic boundary of $M_{\calT}$ is tessellated by the tetrahedral facets of the various copies of $\calR$, and the volume of $M_{\calT}$ is equal to \begin{equation}\label{eq:manifoldvolume}2n\cdot v_{\mathcal{R}}=n\cdot4\pi^2/9.\end{equation}  
\end{proof}

\begin{rem}\label{rem:volumebound}
It is a well-known fact that the volume of a hyperbolic $4$-manifold $\calY$ (possibly with totally geodesic boundary) is proportional to the Euler characteristic of $\calY$ through the relation \begin{equation}\label{eq:volumeeq}\text{Vol}\; \calY=4\pi^2/3\cdot \chi(\calY). \end{equation}
Comparing this formula with (\ref{eq:manifoldvolume}) we see that the number of simplices in a $6$-valent $4$-dimensional triangulation with trivial face cycles is necessarily a multiple of $6$. 
\end{rem}

\begin{rem}\label{rem:cusps}The cusps of $M_{\calT}$ are in one-to-one correspondence with the \emph{edges} of the triangulation $\calT$ and, as shown in the proof of Proposition \ref{prop:hyperbolicstructure} the cusp section are interval-bundles over a flat surface (either the $2$-dimensional torus or the Klein bottle) tessellated by equilateral triangles. \end{rem}

\begin{prop}The boundary components of $M_{\calT}$ are tetrahedral manifolds, in one-to-one correspondence with the \emph{vertices} of the triangulation $\calT$. Each boundary component is triangulated by the link of the corresponding vertex.\end{prop}
\begin{proof} The one-to-one correspondence between the boundary components of $M_{\calT}$ and the vertices of $\calT$ is a straightforward consequence of Remark \ref{rem:correspondence}.
Consider a boundary component $N$ corresponding to a vertex $v$ of $\calT$. The link $L(v)$ of the vertex $v$ in $\calT$  is tessellated by Euclidean tetrahedra. The boundary component $N$ is obtained by removing the vertices of the tessellation of $L(v)$ and realizing each tetrahedron as a regular ideal hyperbolic tetrahedron.
\end{proof}

A consequence of the construction above is that we can find many tetrahedral manifolds which are geodesically embedded:
\begin{prop}\label{prop:embed}
Let $N$ be a complete, finite-volume, hyperbolic tetrahedral $3$-manifold. If $N$ can be realised as the link of a vertex in a $6$-valent $4$-dimensional triangulation $\calT$ with trivial return maps, then $N$ is geodesically embedded.
\end{prop}

\begin{proof}
The $3$-manifold $N$ is isometric to a totally geodesic boundary component of the hyperbolic $4$-manifold $M_{\calT}$. Simply double $M_{\calT}$ in its boundary to obtain a hyperbolic $4$-manifold in which $N$ is geodesically embedded.
\end{proof}

\begin{example}\label{ex:block}
Consider the labeled graph $K_6$ of Figure \ref{fig:graph}.
\begin{figure}[htbp]
\centering
\includegraphics[width=0.35\textwidth]{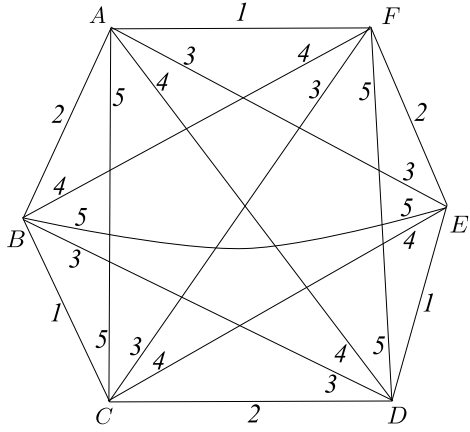}
\caption{The complete graph on six vertices $K_6$ with a $5$-coloring}\label{fig:graph}
\end{figure}
Build a triangulation by taking six copies of $S_4$, in one-to-one correspondence with the vertices of $K_6$. Label the facets of each copy of $S_4$ with the numbers $1,2,3,4,5$. For every edge $e$ of $K_6$ with label $i\in \{1,2,3,4,5\}$ connecting vertices $v$ and $w$, pair the tetrahedral facets with label $i$ of the simplices corresponding to $v$ and $w$ using the identity map, and call the resulting non-orientable triangulation $\calT$.

It is easy to see that the triangulation $\calT$ is $6$-valent and that all the return maps are trivial. The resulting manifold $M_{\calT}$ is isometric to the building block $\calB$ introduced in \cite{KS2014}. It has five tetrahedral boundary components and $10$ cusps. The boundary components are all isometric to each other. Their orientable double cover is the complement of a link in the $3$-sphere, depicted in \cite[p. 148]{ASh}, at the entry $n=4$, $\sigma(n)=6$. Also, the cusp sections of $M_{\calT}$ are all isometric to the product $K\times I$, where $I$ is the unit interval and $K$ is a Klein bottle tessellated by six equilateral triangles.
\end{example}

\section{The figure-eight knot complement}\label{sec:fig8bounds}
In this section we describe a $6$-valent $4$-dimensional triangulation which realises the figure-eight knot complement as a vertex link.
The figure-eight knot complement admits an ideal triangulation with two tetrahedra, $4$ triangular faces, $2$ edges and one vertex, represented in Figure \ref{fig:triangulation}.

\begin{figure}[htbp]
\centering
\includegraphics[width=0.4\textwidth]{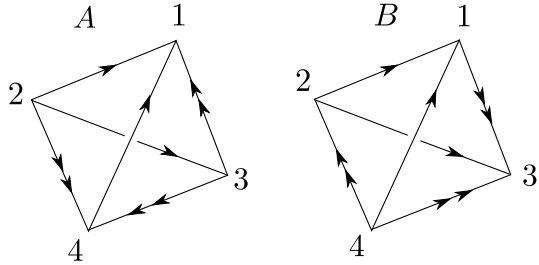}
\caption{The ideal triangulation of the figure eight-knot complement. There is a unique way to pair the faces of the tetrahedron $A$ to the faces of the tetrahedron $B$ respecting the labels on the edges.}\label{fig:triangulation}
\end{figure}

If we label the vertices of the tetrahedra with the numbers $1,2,3,4$ as in Figure \ref{fig:triangulation}, the pairing maps are as follows:
\begin{eqnarray}\label{eq:3triangulation}
A& &B\\
(1,2,4)&\leftrightarrow&(1,4,2)\\
(1,2,3)&\leftrightarrow&(3,2,1)\\
(1,4,3)&\leftrightarrow&(3,2,4)\\
(2,3,4)&\leftrightarrow&(4,1,3).
\end{eqnarray}

Let us begin by defining a $4$-dimensional triangulation consisting of $2$ copies of the simplex $S_4$ with two unpaired tetrahedral facets. Let us take two copies $A$ and $B$ of $S_4$ with vertices labeled by $1,2,3,4,5$, viewed as the cones over the tetrahedra with vertices $1,2,3,4$. We identify them along the faces which are adjacent to the vertex $5$ in such a way as to produce the cone over the triangulation of the figure-eight knot complement:
\begin{eqnarray}\label{eq:4triangulation}
A& &B\\
(5,1,2,4)&\leftrightarrow&(5,1,4,2)\\
(5,1,2,3)&\leftrightarrow&(5,3,2,1)\\
(5,1,4,3)&\leftrightarrow&(5,3,2,4)\\
(5,2,3,4)&\leftrightarrow&(5,4,1,3).
\end{eqnarray}

Let us call $\calC$ the resulting $4$-dimensional triangulation. The two unpaired facets have labels $(1,2,3,4)$. One is a facet of the $4$-simplex $A$ and the other of $B$. Notice that all the pairings map the vertex of $A$ with label $5$ to the vertex of $B$ with label $5$. This results in a vertex $v$ of $\calC$. The link of $v$ is clearly given by the triangulation of the figure-eight knot of Figure \ref{fig:triangulation}. It is clear that the face cycles of $\calC$ corresponding to the triangular $2$-faces of $A$ and $B$ which contain the vertex $5$ have length $6$: this is a direct consequence of the fact that the edges of the triangulation of the figure-eight knot have valence $6$. 

Now let us take $3$ copies of the triangulation $\calC$, labelled $\calX$, $\calY$, $\calZ$. We proceed by glueing them together along the unpaired tetrahedral facets in order to form a cycle as in Figure \ref{fig:ciclo}:

\begin{figure}[htbp]
\centering
\includegraphics[width=0.2\textwidth]{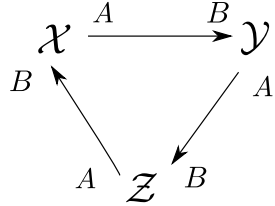}
\caption{Pairing maps between the copies $\calX$, $\calY$, $\calZ$ of $\calC$. We always pair a tetrahedral facet coming from the simplex $A$ with the facet coming from $B$ according to the arrows.}\label{fig:ciclo}
\end{figure}

As a pairing map, we alway use the transposition on the vertices with labels $1$ and $2$: \begin{equation}\label{eq:pairing}A:(1234)\rightarrow B:(2134). \end{equation}
Let us call $\calT$ the resulting triangulation. We already know that the triangles which contain the vertices with label $5$ are organized into cycles of $6$ elements. In order to prove that $\calT$ is $6$-valent, we must check that this holds also for the triangles whose vertices have labels in the set $\{1,2,3,4\}$. The four resulting cycles are the following:

\begin{align}
(1,2,3)_{\calX_{A}}\rightarrow(2,1,3)_{\calY_{B}}\rightarrow(2,3,1)_{\calY_{A}}\rightarrow(1,3,2)_{\calZ_{B}}\rightarrow(3,1,2)_{\calZ_{A}}\rightarrow
(3,2,1)_{\calX_{B}}\rightarrow(1,2,3)_{\calX_{A}}\\
(1,2,4)_{\calX_{A}}\rightarrow(2,1,4)_{\calY_{B}}\rightarrow(4,1,2)_{\calY_{A}}\rightarrow(4,2,1)_{\calZ_{B}}\rightarrow(2,4,1)_{\calZ_{A}}\rightarrow
(1,4,2)_{\calX_{B}}\rightarrow(1,2,4)_{\calX_{A}}\\
(1,4,3)_{\calX_{A}}\rightarrow(2,4,3)_{\calY_{B}}\rightarrow(4,3,1)_{\calY_{A}}\rightarrow(4,3,2)_{\calZ_{B}}\rightarrow(3,1,4)_{\calZ_{A}}\rightarrow
(3,2,4)_{\calX_{B}}\rightarrow(1,4,3)_{\calX_{A}}\\
(2,3,4)_{\calX_{A}}\rightarrow(1,3,4)_{\calY_{B}}\rightarrow(3,4,2)_{\calY_{A}}\rightarrow(3,4,1)_{\calZ_{B}}\rightarrow(4,2,3)_{\calZ_{A}}\rightarrow
(4,1,3)_{\calX_{B}}\rightarrow(2,3,4)_{\calX_{A}}
\end{align}

All the face cycles have length $6$, therefore the triangulation $\calT$ is $6$-valent. Moreover, all the return maps are trivial.
Notice that the pairing maps are of two types: those which extend the triangulation of the figure eight knot complement pair together the two copies $A$ and $B$ of the simplex $S_4$ in each copy of $\calC$, and these alternate with those defined by (\ref{eq:pairing}), which pair together the three different copies of $\calC$. 

\begin{rem}
The triangulation $\calT$ defined above is orientable. Therefore the associated manifold $M_{\calT}$ is orientable.
\end{rem}

\begin{lemma}\label{lemma:bound}
Let $\calM$ be a hyperbolic $n$-manifold which admits a fixed-point-free orientation reversing involution $i$. Suppose that $\calM$ is geodesically embedded in an $(n+1)$-manifold $\calX$. Then $\calM$ bounds geometrically.
\end{lemma}

\begin{proof}
Let us cut the manifold $\calX$ along $\calM$, to produce a manifold $\calX^{\prime}$ with $2$ totally geodesic boundary components homeomorphic to $\calM$. We can ``kill'' one of the boundary components of $\calX^{\prime}$ by taking its quotient under the involution $i$. The resulting orientable manifold $\calY$ has a unique totally geodesic boundary component homeomorphic to $\calM$. Therefore $\calM$ geometrically bounds.
\end{proof}

Now we are finally able to prove the main result of the paper:
\begin{teo}\label{teo:fig8bounds2}
The figure-eight knot complement geometrically bounds a complete, orientable, hyperbolic $4$-manifold $\calY$ with $\chi(\calY)=2$.
\end{teo} 

\begin{proof}
The figure-eight knot complement $\calM$ is geodesically embedded in the double $\calX$ of the manifold $M_{\calT}$, where $\calT$ is the $6$-valent triangulation defined above.  
Moreover it admits a fixed-point-free orientation-reversing involution, since it is the orientable double cover of the Giseking manifold. By applying Lemma \ref{lemma:bound} we see that $\calM$ geometrically bounds an orientable, hyperbolic $4$-manifold $\calY$ which is tessellated by $12$ copies of the polytope $\calR$. The volume of $\calY$ is equal to $12\cdot v_{\calR}=8\pi^2/3$, and by (\ref{eq:volumeeq}) we see that necessarily $\chi(\calY)=2$.
\end{proof}

\begin{rem}\label{rem:minvolume}
It is possible to improve the statement of Theorem \ref{teo:fig8bounds2} and prove that the figure-eight knot complement bounds a complete, orientable, hyperbolic $4$-manifold $\calZ$ with $\chi(\calZ)=1$. To see this, consider the manifold $M_{\calT}$, where $\calT$ is the triangulation defined in Section \ref{sec:fig8bounds}. The boundary $\partial M_{\calT}$ decomposes into $3$ copies of the figure eight-knot complement and a fourth component $N$ which is tessellated by $24$ copies of the regular ideal tetrahedron.

This fourth boundary component is homeomorphic to the manifold \verb otet24_00263  from the Platonic census \cite{Goerner} shipped with the upcoming SnapPy version 2.4 \cite{SnapPy}. This manifold is the orientable double cover of the manifold \verb ntet12_00000  from the same census. Therefore we can produce $\calZ$ by identifying two of the three copies of the figure-eight knot complement in $\partial M_{\calT}$ using an orientation-reversing isometry, and killing the boundary component $N$ using a fixed-point-free orientation reversing involution.

\end{rem}

We conclude by asking the following natural question:
\begin{quest}\label{quest}
Does every tetrahedral $3$-manifold possess an ideal triangulation which is realised as a vertex link of some $6$-valent $4$-dimensional triangulation with trivial return maps?
\end{quest}
A positive answer to this question would readily imply that all tetrahedral $3$-manifolds are geodesically embedded in some hyperbolic $4$-manifold.

\medskip

\begin{flushleft}
\textit{Leone Slavich\\
Dipartimento di Matematica \\
Largo Bruno Pontecorvo 5\\
56127 Pisa, Italy\\}
\texttt{leone dot slavich at gmail dot com}
\end{flushleft}

\end{document}